\documentclass[12pt]{amsart}
\usepackage{amssymb,latexsym,verbatim}

\theoremstyle{plain}
\newtheorem{thm}{Theorem}[section]
\newtheorem{lem}[thm]{Lemma}
\newtheorem{cor}[thm]{Corollary}
\newtheorem{prop}[thm]{Proposition}

\theoremstyle{definition}

\theoremstyle{remark}
\newtheorem{exam}[thm]{Example}

\DeclareMathOperator{\Ann}{Ann}

\DeclareMathOperator{\Ext}{Ext}
\DeclareMathOperator{\Supp}{Supp}
\DeclareMathOperator{\V}{V}
\DeclareMathOperator{\Hom}{Hom}
\DeclareMathOperator{\Ker}{Ker}

\DeclareMathOperator{\depth}{depth}
\DeclareMathOperator{\cd}{cd}
\DeclareMathOperator{\q}{q}
\DeclareMathOperator{\Ass}{Ass}
\DeclareMathOperator{\Coass}{Coass}

\DeclareMathOperator{\Max}{Max}
\DeclareMathOperator{\LC}{H}

\DeclareMathOperator{\Spec}{Spec}
\DeclareMathOperator{\G}{\Gamma}
\DeclareMathOperator{\Att}{Att}
\DeclareMathOperator{\E}{E}

\newcommand{\fa}{\mathfrak{a}}

\newcommand{\fm}{\mathfrak{m}}
\newcommand{\fp}{\mathfrak{p}}

\begin{document}

\title[Local cohomology: vanishing, artinianness and finiteness ]
{On the vanishing, artinianness and finiteness of local cohomology
modules\\}

\author{Moharram Aghapournahr}
\address{ Moharram Aghapournahr\\ Arak University\\ Beheshti St, P.O. Box:879,
Arak, Iran}
\email{m-aghapour@araku.ac.ir}

\author{Leif Melkersson}
\address{Leif Melkersson\\Department of Mathematics
\\ Link\"{o}ping University\\ S-581 83 Link\"{o}ping\\ Sweden}

\email{lemel@mai.liu.se}

\keywords{Local cohomology, minimax module, coatomic module.\\}

\subjclass[2000]{13D45, 13D07}


\begin{abstract}
 Let $R$ be a noetherian ring, $\fa$ an ideal of $R$, and  $M$ an $R$--module.
 We prove that for a finite module $M$, if  $\LC^{i}_{\fa}(M)$ is
minimax for all
$i\geq r\geq 1$, then $\LC^{i}_{\fa}(M)$ is artinian
 for $i\geq r$.
A Local-global Principle for minimax local cohomology modules is shown.
If $\LC^{i}_{\fa}(M)$ is coatomic for $i\leq r$ ($M$ finite)
then $\LC^{i}_{\fa}(M)$ is finite for $i\leq r$. We give conditions for
a module, which is locally minimax to be a minimax module.
A non-vanishing theorem and some vanishing theorems are proved for
local cohomology modules.
\end{abstract}

\maketitle

\section{Introduction}

Throughout $R$ is a commutative noetherian ring. For unexplained
items from homological and commutative algebra we refer to \cite{BH}
and \cite{Mat}.

Huneke gave in \cite{Hu} a survey of some important problems on
finiteness, vanishing and artinianness of local cohomology modules.
We give some further contributions to the study of certain
finiteness, vanishing and artinianness
 results for the local cohomology modules $\LC^{i}_{\fa}(M)$ for an
$R$--module $M$ with respect to an ideal $\fa$.
A thorough treatment of local cohomology is given by
Brodmann and Sharp in \cite{BSh}.

A module $M$ is a \emph{minimax} module if there is a finite (i.e.
finitely generated) submodule $N$ of $M$ such that the quotient
module $M/N$ is artinian. Thus the class of minimax modules includes
all finite and all artinian modules.
Moreover, it is closed under taking submodules, quotients and extensions,
i.e., it is a Serre subcategory of the
 category of $R$--modules.
Minimax modules have been studied by Zink in \cite{Zi} and  Z\"{o}schinger in
 \cite{Zrmm,Zrrad}. See also \cite{Ru}.
Many equivalent conditions for a module to be minimax are given by them.
  We summarize some of those as follows:
\begin{thm}\label{T:mmeq}
For a module $M$ over the commutative noetherian ring $R$, the
following conditions are equivalent:
\begin{enumerate}
  \item[(i)] $M/N$ has finite Goldie dimension for each submodule $N$ of $M$.
  \item[(ii)] $M/N$ has finite socle for each submodule $N$ of $M$.
  \item[(iii)] $M/N$ is an artinian module whenever
    $N$ is a submodule of $M$, such that $\Supp_R(M/N)\subset \Max{R}$.
  \item[(iv)] $M/N$ is artinian for some finite submodule $N$ of $M$.
  \item[(v)] For each increasing sequence
    $N_1 \subset N_2 \subset\dots $ of submodules of $M$ there is $l$ such that
   $ N_{n+1}/N_n$ is artinian for all $n\geq l$.
  \item[(vi)] For each decreasing sequence
    $N_1 \supset N_2 \supset\dots$ of submodules of $M$ there is $l$ such that
   $ N_{n+1}/N_n$ is finite for all $n\geq l$.
  \item[(vii)] {\rm{(}}When $(R,\fm)$ is a complete local ring{\rm{)}}
    $M$ is Matlis reflexive.
\end{enumerate}
\end{thm}

An $R$--module $M$ has \emph{finite Goldie dimension} if $M$ contains
no infinite direct sum of submodules. For a commutative noetherian
ring this can be expressed in two other ways, namely that the
injective hull $\E(M)$ of $M$ decomposes as a finite direct sum of
indecomposable injective modules or that $M$ is an essential
extension of a finite submodule. In \ref{P:mmwl} we will give
another equivalent condition for a module to be minimax.

We prove in \ref{T:mmart}, that when $M$ is a finite $R$--module
such that the local cohomology modules $\LC^{i}_{\fa}(M)$ are
minimax modules for all $i\geq r$, where $r\geq 1$ then they must
be artinian.

An $R$--module $M$ is called  $\fa$--\emph{cofinite} if
$\Supp_R(M)\subset \V{(\fa)}$ and $\Ext^i_{R}(R/\fa,M)$ is finite
for each $i$. Hartshorne introduced this notion in \cite{Ha}, where
he gave a negative answer to a question by Grothendieck in
\cite{SGA2}, by giving an example of a local cohomology module which
is not $\fa$--cofinite. If an $R$--module $M$ with support in
$\V{(\fa)}$ is known to be a minimax module, then it suffices to
know that $0:_M{\fa}$ is finite in order to conclude that $M$ is
$\fa$--cofinite, \cite[Proposition 4.3]{Mel}. If we know that
$0:_M{\fa}$ is finite, then of course in general $M$ is neither
minimax nor $\fa$--cofinite, but if $M$ is assumed to be locally
minimax, then $M$ is $\fa$--cofinite and minimax as we show in
\ref{T:lmmcof}. This is applied to prove a Local-global Principle
for minimax modules in \ref{T:lgpmm}.

A prime ideal $\fp$ is said to be \emph{coassociated} to $M$ if
$\fp=\Ann_R({M/N})$ for some $N\subset M$ such that
 $M/N$ is artinian and is said to be \emph{attached} to
  $M$ if $\fp=\Ann_R({M/N})$ for some arbitrary submodule
 $N$ of $M$, (equivalently $\fp=\Ann_R({M/{\fp}M})$).
The set of these prime ideals are denoted by $\Coass_R(M)$
  and $\Att_R(M)$ respectively.
Thus $\Coass_R(M)\subset \Att_R(M)$ and the two sets are equal when $M$ is an
  artinian module.
An alternative description for coassociated primes is given by

$$
\Coass_R(M)=\underset{\fm\in
\Max{R}}{\bigcup}\Ass_R(\Hom_{R}(M,\E{(R/\fm)})).
$$
Thus when $(R,\fm)$ is a local ring the coassociated primes of an
$R$--module are just the associated primes of its
 Matlis dual.

$M$ is called \emph{coatomic} when each proper submodule $N$ of $M$
is contained in a maximal submodule $N^{\prime}$
 of $M$ (i.e. such that $M/N^{\prime}\cong R/\fm$ for some
$\fm \in \Max{R}$). This property can also
be expressed by $\Coass_R(M)\subset \Max{R}$ or equivalently that
any artinian homomorphic image of $M$ must have finite length. In
particular all finite modules are coatomic. Coatomic modules have
been studied by Z\"{o}schinger \cite{Zrko}.

A module $M$ which is minimax or coatomic has the property that the
localization $M_{\fp}$ is a finitely generated $R_{\fp}$--module for
each non-maximal prime ideal $\fp$. When $M$ is a minimax module
this follows from  condition (iv) of\ref{T:mmeq}.

We show in \ref{T:kofinlc} that if $M$ is finite and all local
cohomology modules $\LC^{i}_{\fa}(M)$ are coatomic for all $i<n$,
then they are actually finite in this range. In fact this is another
condition equivalent to Falting's Local-global Principle for the
finiteness of local cohomology modules, \cite[Theorem 9.6.1 and
Proposition 9.1.2]{BSh}. A vanishing theorem of Yoshida \cite{Yo} is
generalized in \ref{T:vanko1} and \ref{C:vanko2}.

For an $R$--module $M$ and an ideal $\fa$ of $R$, we let
$$
\cd{(\fa,M)}=\min\{n\geq 0\mid\LC^{i}_{\fa}(M)=0 \text{ for all } i>n \,\}
$$
and
$$
\q(\fa,M)=\min\{n\ge 0\mid\LC^{i}_{\fa}(M) \text{ is artinian for
all } i>n \,\}.
$$

We show that if $M$ is a coatomic $R$--module, then for any
$R$--module $N$ such that $\Supp_R(N)\subset\Supp_R(M)$,
 we have $\cd{(\fa,N)}\leq \cd{(\fa,M)}$.
This generalizes a result by Dibaei and Yassemi in \cite[Theorem 1.4]{DY3}
  who proved it when $M$ is finite.

\section{Artinianness of local cohomology modules}

\begin{lem}\label{L:lart}

Let $M$ be an $R$--module. The following statements are equivalent:
\begin{enumerate}
  \item[(i)]$M$ is an artinian $R$--module.
  \item[(ii)]$M_\fm$ is an artinian $R_\fm$--module for all
$\fm\in{\Max{R}}$ and $\Ass_R(M)$ is a finite set.
\end{enumerate}
 \end{lem}

A module $M$ is \emph{weakly laskerian} when each quotient $M/N$ has
just finitely many associated primes.  For a study of such modules,
see \cite{DM}. Every minimax
module is trivially weakly laskerian. The converse holds under the
additional condition that the module is locally minimax.

\begin{prop}\label{P:mmwl}
Let $M$ be an $R$--module. The following statements are equivalent:
\begin{enumerate}
  \item[(i)]$M$ is a minimax $R$--module.
  \item[(ii)]$M_\fm$ is a minimax $R_\fm$--module for all
    $\fm{\in}\Max{R}$ and $M$ is a weakly laskerian
  $R$--module.
\end{enumerate}
\end{prop}

\begin{proof}
The only nontrivial part is (ii)$\Rightarrow$ (i).

 We show that if $\Supp_R(M/N)\subset \Max{R}$ then $M/N$ is artinian.
By hypothesis $\Ass_R(M/N)$ is a finite set
  and consists of maximal ideals.
For each maximal ideal $\fm$,
the $R_\fm$--module $(M/N)_\fm$ is a minimax module
   with support at the maximal ideal of $R_\fm$.
Therefore by part (iii) of \ref{T:mmeq} $(M/N)_\fm$ is an artinian
    $R_\fm$--module for all $\fm\in{\Max{R}}$.
By \ref{L:lart}, $M/N$ is an artinian $R$--module.
\end{proof}

The following theorem is the main result of this section.

\begin{thm}\label{T:mmart}
Let $R$ be a noetherian ring, $\fa$ an ideal of $R$ and $M$ a finite
$R$--module. If $r \ge 1$ is an integer such that
$\LC^{i}_{\fa}(M)$ is a minimax module for all $i \ge r$, then
$\LC^{i}_{\fa}(M)$ is an artinian module for all $i \ge r$.
\end{thm}
\begin{proof}

Suppose $\fp$ is a nonmaximal prime ideal of $R$. Then
$\LC^{i}_{\fa}(M)_\fp \cong \LC^{i}_{{\fa}R_{\fp}}(M_\fp)$ is a
finite $R_\fp$--module for all $i\geq r$, since as we remarked in the
introduction, when we localize at nonmaximal prime ideals, we obtain
finitely generated modules.
Therefore from \cite[Proposition 3.1]{Yo} we get that
$\LC^{i}_{\fa}(M)_\fp=0$ for all $i\ge r$.
Hence $\Supp_R(\LC^{i}_{\fa}(M))\subset{\Max{R}}$ for all $i\ge r$.
By the condition (iii) of \ref{T:mmeq}, the modules
$\LC^{i}_{\fa}(M)$ are artinian
for all $i \geq r$.
\end{proof}


\begin{cor}\label{C:mmart}

Let $\fa$ an ideal of $R$ and $M$ a finite $R$--module. If
$q=\q(\fa,M)> 0$, then the module $\LC^{q}_{\fa}(M)$ is not
minimax, in particular it is not finite.
\end{cor}

\begin{prop}\label{P:mmcsoc}
Let $M$ be a minimax module and $\fa$ an ideal of $R$. If $M$ is
$\fa$--cofinite and socle-free, then there is $l$ such that
$M=0:_M{\fa^l}$ and $M$ is finite.
\end{prop}

\begin{proof}
Given $n$, let $\fa^n=(c_1,\dots,c_r)$. We define $h:M\to M^r$ by
$h(m)=({c_i}m)_{i=1}^r$. Clearly $\Ker{h}=0:_M{\fa^n}$, so the
module $M_n=M/(0:_M{\fa^n})$ is isomorphic to a submodule of $M^r$.
In particular $M_n$ is socle-free. Consider the increasing sequence
$0:_M{\fa}\subset 0:_M{\fa^2}\subset\dots$ of
 submodules of $M$, whose union is equal to $M$.
Since $M$ is minimax, \ref{T:mmeq} (v) implies that there is $l$
 such that $0:_M{\fa^{n+1}}/(0:_M{\fa^n})$ is artinian for all
$n\geq l$. But $M/(0:_M{\fa^n})$ is socle-free.
 Hence $0:_M{\fa^{n+1}}=0:_M{\fa^n}$ for all $n\geq l$,
and therefore $M=0:_M{\fa^l}$.
\end{proof}

The following theorem generalizes \cite[Proposition 4.3]{Mel}.

\begin{thm}\label{T:lmmcof}
 Let $M$ be an $R$--module such that
$\Supp_R(M)\subset \V{(\fa)}$ and $M$ is locally minimax.
If $0:_M{\fa}$ is finite, then $M$ is an $\fa$--cofinite minimax
module. In particular this is the case, if there exists an element
$x\in \fa$ such that $0:_M{x}$ is $\fa$--cofinite.
\end{thm}

\begin{proof}
Let $L$ be the sum of the artinian submodules of $M$. Then
$0:_L{\fa}$ is finite and therefore has finite length.
 Hence by \cite[Proposition 4.1]{Mel} $L$ is artinian and $\fa$--cofinite.

The module $\overline{M}=M/L$ is locally minimax and furthermore it
is socle-free. From the exactness of
$$
0\to 0:_L{\fa}\to 0:_M{\fa}\to 0:_{\overline{M}}{\fa}\to \Ext^1_{R}(R/\fa,L),
$$
we get that $0:_{\overline{M}}{\fa}$
is finite. We may therefore replace $M$ by $\overline{M}$, and
assume that $M$ is socle-free.

Let $\fm$ be any maximal ideal. Then $M_{\fm}$ is a socle-free
minimax module over $R_{\fm}$, in fact it is ${\fa}R_\fm$--cofinite
by \cite[Proposition 4.3]{Mel}. We are therefore able to apply
proposition \ref{P:mmcsoc}, so there is $n$ such that $(M_n)_\fm=0$
where $M_n=M/(0:_M{\fa}^n)$. Since as noted in the proof of
\ref{P:mmcsoc} for each $n$, there is $r$ such that $M_n$ is
isomorphic to a submodule of $M^r$, $\Ass_R(M_n)\subset
\Ass_R(M)=\Ass_R(0:_M{\fa})$ and $\Ass_R(M_n)$ is therefore finite.
Consequently $\Supp_R(M_n)$ must be a closed subset of $X=\Spec R$.
Therefore
$U_n=X\setminus{\Supp_R(M_n)}$ is an increasing sequence of open
subsets of $X$. Since for each maximal ideal $\fm$, there is $n$
such that $(M_n)_\fm=0$ i.e. $\fm \in{U_n}$,
$X=\cup_{n=0}^\infty{U_n}$. By the quasi-compactness of $X$, we get
that $X=U_n$ for some $n$. Hence $M=0:_M{\fa}^n$ which is finite.
\end{proof}

The following corollary describes a relation between the properties
of cofiniteness and minimaxness for local cohomology.

\begin{cor}\label{C:lmmcof}
Let $n$ be a non-negative integer and $M$ a finite $R$--module
\begin{enumerate}
\item[(a)]
If $\LC^{i}_{\fa}(M)$ is $\fa$--cofinite for all $i<n$ and
$\LC^{t}_{\fa}(M)$ is a locally minimax module, then it is also
$\fa$--cofinite minimax.
\item[(b)] If $\LC^{i}_{\fa}(M)$ is
$\fa$--cofinite for all $i<n$ and a locally minimax
 module for all $i$ $\geq$ $n$,
then $\LC^{i}_{\fa}(M)$ is $\fa$--cofinite for all $i$.
\end{enumerate}
\end{cor}

\begin{proof}
 It is enough to prove that $\Hom_{R}(R/\fa,\LC^{n}_{\fa}(M))$ is
finite by \ref{T:lmmcof} and this is immediate by use of
\cite[Theorem 2.1]{DY2}

(b) Use part (a).
\end{proof}
The following theorem, which is one of our main results shows that
the Local-global Principle is valid for
 minimax local cohomology modules.

\begin{thm}\label{T:lgpmm}
 Let $\fa$ be an ideal of $R$, $M$ a finite $R$--module and
$t$ a non-negative integer. The following statements
  are equivalent:
\begin{enumerate}
 \item[(i)] $\LC^{i}_{\fa}(M)$ is a minimax $R$--module for all $i\le t$.
 \item[(ii)] $\LC^{i}_{\fa}(M)$ is an
   $\fa$--cofinite minimax $R$--module for all $i\le t$.
  \item[(iii)] $\LC^{i}_{\fa}(M)_{\fm}$ is a minimax
    $R_{\fm}$--module for all $\fm\in \Max{R}$ and for all $i\le t$.
\end{enumerate}
\end{thm}

\begin{proof}
The only non-trivial part is the implication
$(iii)\Rightarrow (ii)$. We prove this by induction on $t$. When
$t=0$ there is nothing to prove. Suppose $t>0$ and the case $t-1$ is
settled. So we may assume that $\G_{\fa}(M)=0$. Thus there exists
$x\in {\fa}$ such that
$0\to M\overset{x}\to M\to M/{x}M\to 0$
is exact. We get  the exact sequence
$$
\LC^{i}_{\fa}(M)_{\fm}\to \LC^{i}_{\fa}(M/{x}M)_{\fm}\to
\LC^{i+1}_{\fa}(M)_{\fm}
$$
It follows that $\LC^i_\fa (M/{xM})_\fm$ is a minimax
$R_\fm$--module for $i\le{t-1}$. By the induction hypothesis
$\LC^i_{\fa}(M/{xM})$ is an $\fa$--cofinite $R$--module for $i\le
t$. It follows that $0\underset{\LC^t_\fa(M)}{:} x$ is
$\fa$--cofinite and from \ref{T:lmmcof} we conclude that
$\LC^t_\fm(M)$ is $\fa$--cofinite minimax.
\end{proof}



\begin{exam}
Suppose the set $\Omega$ of maximal ideals of $R$ is infinite. Then
the module $\oplus_{\fm\in\Omega}R/\fm$  is locally a minimax
module, but it is not a minimax module.
\end{exam}

\section{Finiteness, vanishing and non vanishing}

\begin{lem}\label{L:fin}
 Let $R$ be a noetherian ring, $\fa$ an ideal of $R$, $M$ an $R$--module.
Then $\fa M$ is finite if and only
 if $M/(0:_M{\fa})$ is finite.
\end{lem}
\begin{proof}
 $( \Rightarrow)$ Suppose $\fa=(a_1,\dots,a_n)$ and define
$f: M\to (\fa M)^n$ by $f(m)=(a_i m)_{i=1}^n$. Since  $\ker{f}=
0:_M{\fa}$, the module $M/(0:_M{\fa})$ is isomorphic to a submodule
of $(\fa M)^n$.

$(\Leftarrow)$ Define a homomorphism $g:M^n\to \fa M$ by
$g((m_i)_{i=1}^n)=\sum\limits_{i=1}^n{a_i m_i}$. Then $g$ is
surjective and $(0:_M{\fa})^n\subset \Ker{g}$, so $\fa M$ is a
homomorphic image of $(M/(0:_M{\fa}))^n$.

By the way, when $\fa={x}R$ is a  principal ideal,
the modules ${x}M$ and $M/(0:_M{x})$
are in fact isomorphic.
\end{proof}

\begin{thm}[Nonvanishing for coatomic modules]\label{T:nonvan}
 Let $(R,\fm)$ be
a noetherian local ring. If $M$ is a
 nonzero coatomic $R$--module of dimension $n$, then $\LC^{n}_{\fm}(M) \neq 0$.
\end{thm}
\begin{proof}
If $n = 0$, there is nothing to prove. Suppose $n \geq 1$ then from
\cite[Satz 2.4 $(i)\Rightarrow (iii)$]{Zrko} there is an integer $t
\geq 1$ such that $\fm^tM$ is finite and by \ref{L:fin} equivalently
$M/(0:_M{\fm^t})$ is finite. On the other hand $\dim_R
M/(0:_M{\fm^t})=\dim_R M=n$, and $\LC^{i}_{\fm}(0:_M{\fm^t})=0$ for
all $i\geq 1$. Making use of the exact sequence
$$
0\to 0:_M{\fm^t}\to M\to M/(0:_M{\fm^t})\to 0
$$
we get $\LC^{n}_{\fm}(M) \cong \LC^{n}_{\fm}(M/(0:_M{\fm^t}))$,
which is $\ne 0$, by \cite[Theorem 6.1.4]{BSh}.
\end{proof}

\begin{lem}\label{L:mar}[See also \cite[Corollary 2.5]{Mar}.]
If $R$ and $\fa$
are as before and $M$ is a finite $R$--module of dimension $n$, then
\begin{enumerate}
 \item[(a)] $\dim_R \LC^{n-i}_{\fa}(M)\leq i$.
 \item[(b)] If $(R,\fm)$ is a local ring,
then $\Supp_R(\LC^{n-1}_{\fa}(M))$ is a finite set consisting of
 prime ideals $\fp$ such that $\dim{R/\fp}\leq 1$.
\end{enumerate}
\end{lem}
\begin{proof}(a) For $\fp\in \Supp_R(\LC^{n-i}_{\fa}(M))$, we get
$\LC^{n-i}_{\fa}(M)_\fp \cong \LC^{n-i}_{{\fa}R_{\fp}}(M_\fp)\ne
0$.Hence
 \cite[Theorem 6.1.2]{BSh} implies that $\dim{M_\fp}\geq n-i$ and therefore
we have
$$
\dim{R/\fp}\leq n-\dim{M_\fp}\leq i.
$$

(b) Let $\fm=(x_1,\dots,x_r)$. Then $\dim{M_{x_i}}\leq n-1$ for $1\leq
i\leq r$. Hence by \cite[Exercise 7.1.7]{BSh}
$\LC^{n-1}_{\fa{R}_{x_i}}(M_{x_i})$ is an artinian $R_{x_i}$--module and
$\Supp_{R_{x_i}}(\LC^{n-1}_{\fa}(M)_{x_i})$ is finite.
If $\fp\in\Supp_R( \LC^{n-1}_{\fa}(M))$ and $\fp\neq \fm$ then there is $i$
such that ${x_i}\notin \fp$, i.e.
${\fp}R_{x_i}\in\Supp_{R_{x_i}}(\LC^{n-1}_{\fa}(M)_{x_i})$. Hence
$\Supp_R(\LC^{n-1}_{\fa}(M))$ must be finite.
\end{proof}

\begin{prop}\label{P:kolc}
Let $M$ be a coatomic module of dimension $n\geq 1$ over the local
ring $(R,\fm)$ and let $\fa$ be an ideal of $R$. Then we have that
\begin{enumerate}
  \item[(a)] $\LC^{n}_{\fa}(M)$ is artinian and $\fa$--cofinite.
  \item[(b)] $\Att_R(\LC^{n}_{\fa}(M))=
   \{\fp\in{\Supp_R(M)}|\cd{(\fa,R/\fp)}=n\}$.
  \item[(c)] $\Supp_R(\LC^{n-1}_{\fa}(M))$ is a finite set
   consisting of prime ideals $\fp$ such that
   $\dim{R/\fp}\leq 1$.
\end{enumerate}
\end{prop}
\begin{proof}
(a): As in the proof of \ref{T:nonvan} we have
\begin{equation}\label{E:iso}
\LC^{n}_{\fa}(M) \cong \LC^{n}_{\fa}(M/(0:_M{\fm^t}))
\end{equation}
for some $t\geq 1$ such that $M/(0:_M{\fm^t})$ is a finite
$R$--module. Consequently by \cite[Proposition 5.1]{Mel}
$\LC^{n}_{\fa}(M)$ is artinian and $\fa$--cofinite.

(b): Put $L=M/(0:_M{\fm^t})$ and note that $\Supp_R(L)=\Supp_R(M)$.
But by \cite[Theorem A]{DY1}
$$
\Att_R(\LC^{n}_{\fa}(L))=\{\fp\in{\Supp_R(L)}|\cd{(\fa,R/\fp)}=n\},
$$
so the assertion holds.
(c). Use the isomorphism~(\ref{E:iso}) and part (b) of \ref{L:mar}.
\end{proof}
However when $n=0$, $\LC^{n}_{\fa}(M)$ may not be artinian.
\begin{exam}\label{E:koart}
 $M={(R/\fm)}^{(\mathbb N)}$ is an
$\fm$--torsion coatomic module of dimension zero but is not artinian.
\end{exam}

\begin{lem}\label{L:cd}
Let $M$ be a finite or more generally a coatomic $R$--module. Then
$\cd{(\fa,M)}=0$ if and only if $\Supp_R(M)\subset \V{(\fa)}$.
\end{lem}
\begin{proof}
$(\Leftarrow)$. Trivial.

$(\Rightarrow)$. We may assume that $(R,\fm)$ is local. First assume
that $M$ is finite.

If $\Supp_R(M)\not\subset \V{(\fa)}$, then the module
$\overline{M}=M/\G_{\fa}(M)$ is nonzero, and
$\G_{\fa}(\overline{M})=0$. Hence we have
$r=\depth_{\fa}\overline{M}>0$, but by \cite[Theorem 6.2.7]{BSh}
$\LC^r_\fa(\overline{M})\neq 0$. On the other hand
 $\LC^r_\fa(M)\cong \LC^r_\fa(\overline{M})$ and this is a contradiction.

Now suppose $M$ is coatomic. As before for any $r>0$ we have
$\LC^{r}_{\fa}(M) \cong \LC^{r}_{\fa}(M/(0:_M{\fm^t}))$ for some
$t\geq 1$ such that $M/(0:_M{\fm^t})$ is finite. Note that
$\Supp_R(M/(0:_M{\fm^t}))=\Supp_R(M)$ and use the result just shown
for finite modules.
\end{proof}

We next  generalize \cite[Theorem 1.4]{DY3}. See also
\cite[Theorem 2.2]{DNT}.
\begin{prop}\label{P:cd}
Let $\fa$ be an ideal of $R$ and $M$ a coatomic $R$--module. Let $N$
be an arbitrary module such that
 $\Supp_R(N)\subset\Supp_R(M)$, then $\cd{(\fa,N)}\leq \cd{(\fa,M)}$
\end{prop}

\begin{proof}
We may assume that $(R,\fm)$ is local. Suppose $\cd{(\fa,M)}=0$,
then by \ref{L:cd} $\Supp_R(M)\subset \V{(\fa)}$. Hence
$\Supp_R(N)\subset \V{(\fa)}$ and therefore $\LC^{i}_{\fa}(N)=0$ for
all $i>0$, i.e. $\cd{(\fa,N)}=0$.

Let $\cd{(\fa,M)}\geq 1$, Then as before $\LC^{r}_{\fa}(M) \cong
\LC^{r}_{\fa}(M/(0:_M{\fm^t}))$ for some $t\geq 1$ such that
$M/(0:_M{\fm^t})$ is finite. Since
$$
\Supp_R(N)\subset\Supp_R(M)=\Supp_R(M/(0:_M{\fm^t})),
$$
we get from \cite[Theorem 1.4]{DY3}
$$
\cd{(\fa,N)}\leq \cd{(\fa,M/(0:_M{\fm^t}))}=\cd{(\fa,M)}.
$$
\end{proof}

Next we prove some vanishing and finiteness results for local
cohomology.

\begin{thm}\label{T:kofinlc}
Let $R$ be a noetherian ring, $\fa$ an ideal of $R$ and $M$ a finite
$R$--module. The following statements are equivalent:
\begin{enumerate}
  \item[(i)] $\LC^{i}_{\fa}(M)$ is coatomic for all $i<n$.
  \item[(ii)] $\Coass_R(\LC^i_\fa(M)) \subset \V{(\fa)}$ for all $i<n$.
  \item[(iii)] $\LC^{i}_{\fa}(M)$ is finite for all $i<n$.
\end{enumerate}
\end{thm}
\begin{proof}
By \cite[Theorem 9.6.1]{BSh} and\cite[1.1, Folgerung]{Zrko} we may
assume that $(R,\fm)$ is a local ring.
 \item[(i)]$\Rightarrow$ (ii) It is trivial by the definition of
  coatomic modules.
 \item[(ii)]$\Rightarrow$ (iii) By \cite[Satz 1.2]{Zrkoass} there is
  $t \geq 1$ such that $\fa^t\LC^i_\fa(M)$
 is finite for all $i<n$. Therefore there is $s \geq t$ such that
$\fa^s\LC^i_\fa(M)=0$ for all $i<n$. Then apply \cite[Proposition
9.1.2]{BSh}.
 \item[(iii)]$\Rightarrow$ (i) Any finite $R$--module is coatomic.
\end{proof}

The following results are generalizations of \cite[Proposition 3.1]{Yo}

\begin{thm}\label{T:vanko1}
Let $\fa$ be an ideal of $R$ and $M$ a finite $R$--module and let
$r\ge 1$. The following statements are equivalent:
\begin{enumerate}
 \item[(i)] $\LC^i_\fa(M)=0$ for all $i\ge r$.
 \item[(ii)] $\LC^i_\fa(M)$ is finite for all $i\ge r$.
 \item[(iii)] $\LC^i_\fa(M)$ is coatomic for all $i\ge r$.
\end{enumerate}
\end{thm}

\begin{proof}
$(i)\Rightarrow (ii)\Rightarrow (iii)$ Trivial.

$(iii)\Rightarrow (i)$ By use of \cite[Proposition 3.1]{Yo} and
\cite[1.1, Folgerung]{Zrko} we may assume that $(R,\fm)$ is a local
ring.
 Note that coatomic modules satisfies Nakayama's lemma.
So the proof is the same as in \cite[Proposition 3.1]{Yo}.
\end{proof}

\begin{cor}\label{C:vanko2}
 Let $M$ be a coatomic $R$--module.
If $\LC^i_\fa(M)$ is coatomic for all $i\geq r$, where $r\geq 1$,
  then $\LC^i_\fa(M)=0$ for all $i\geq r$.
\end{cor}
\begin{proof}
 We may assume that $(R,\fm)$ is a local ring.
So as before there is an isomorphism
$$
\LC^{r}_{\fa}(M) \cong \LC^{r}_{\fa}(M/(0:_M{\fm^t}))
$$
 for some $t\ge 1$ such that $M/(0:_M{\fm^t})$ is finite,
and then use  $(iii)\Rightarrow (i)$ of
  \ref{T:vanko1}.
\end{proof}

\begin{cor}\label{C:vanco1}
Let $\fa$ an ideal of $R$ and $M$ a finite $R$--module. If
$c=\cd(\fa,M)> 0$, then $\LC^{c}_{\fa}(M)$ is not coatomic in
particular it is not finite.
\end{cor}

\begin{cor}\label{C:vanko3}
If $M$ is coatomic and $r\geq 1$, the following are equivalent:
\begin{enumerate}
 \item[(i)] $\LC^i_\fa(M)=0$ for all $i\geq r$.
 \item[(ii)] $\LC^i_\fa(M)$ is finite for all $i\geq r$.
 \item[(iii)] $\LC^i_\fa(M)$ is coatomic for all $i\geq r$.
\end{enumerate}
\end{cor}


\end{document}